\documentclass[a4paper]{article}
\usepackage[utf8]{inputenc}
\usepackage{graphicx}
\usepackage{flushend}
\usepackage{amsmath}
\usepackage{amsthm}
\usepackage{nccmath}
\usepackage{amssymb}
\usepackage{enumerate}
\usepackage{subfig}
\usepackage{fancyhdr}
\usepackage[percent]{overpic}

\def\M{\mathbb{M}}
\def\r{\mathbb{R}}
\def\R{\mathbb{R}}
\def\t{\tau}
\def\h{\mathbb{H}}
\def\s{\mathbb{S}}
\def\k{\kappa}
\def\mx{\mathfrak{X}}
\def\mt{\mathbf{T}}
\def\la{\langle}
\def\ra{\rangle}
\def\nt{\mathbf{t}}
\def\te{\theta}
\def\me{\mathbb{M}^2(\kappa)}
\def\mer{\mathbb{M}^2 (\kappa)\times \mathbb{R}}

\newcommand{\zb}{\bar{z}}
\newcommand{\hmf}{\mathbb{E}(\kappa , \tau)}

\newcommand{\abs}[1]{\left\vert #1 \right\vert}
\newcommand{\set}[1]{\left\{#1\right\}}
\newcommand{\meta}[2]{\langle #1,#2 \rangle }
\theoremstyle{plain}
 \newtheorem{definition}{Definition}
 \newtheorem{theorem}{Theorem}
 \newtheorem{example}{Example}
 \newtheorem{prop}{Proposition}
 \newtheorem{remark}{Remark}
 \newtheorem{lema}[prop]{Lema}

 \newtheorem{Corollary}{Corollary}

\begin{document}
\title{Capillary Immersions in $\hmf$.}
\author{Haimer A. Trejos}
\date{}
\maketitle

\begin{abstract}
In \cite{AEG3} and \cite{ER4} the authors showed the existence of a  Codazzi pair defined on any constant mean curvature surface in the homogeneous spaces $\hmf$ associated to the Abresch-Rosenberg differential. In this paper, we use the mentioned Codazzi pair to classify capillary disks in $\hmf$. As a consequence,  the results presented in this paper generalize the previous classification of  constant mean curvature disks in the product spaces $\mathbb{S}^{2} \times \mathbb{R}$ and $\mathbb{H}^{2} \times \mathbb{R}$ in \cite{dCF} and \cite{Cal}. 
\end{abstract}

\section{Introduction.}

Heinz Hopf in \cite{Ho} showed that any immersed constant mean curvature topological sphere in $\mathbb{R}^{3}$ must be a round sphere. To do so, he introduced a quadratic differential on any constant mean curvature surface, the Hopf differential, that turns out to be holomorphic for the conformal structure defined by the first fundamental of the surface and he observed that the zeroes of this differential coincide  with the umbilical points of the surface. Finally, using that any holomorphic quadratic differential defined on a topological sphere must be the trivial one, the proof follows from the classification of umbilical surfaces in $\mathbb{R}^{3}$.

It should be remarked that the Codazzi equation of a constant mean curvature surface (or $H-$ surface) in $\mathbb{R}^{3}$ implies that the Hopf differential is holomorphic. Hence, as the Codazzi equation is the same in every 3-dimensional space form, we can extend the Hopf Theorem to any space form.\\

Regarding $H-$ surfaces with non-empty smooth boundary, J.C. Nistche in \cite{Ni} used the Hopf differential to classify free boundary constant mean curvature disks (or $H-$ disks) in the Euclidean ball of $\mathbb{R}^{3}$. He showed that the boundary is a line of curvature. Then, he used the Schwarz Reflection Principle to concluded that  the Hopf differential vanishes on any free boundary $H-$ disk and hence the disk  must be totally umbilical. 

When the boundary is a piece-wise regular curve, J. Choe in \cite{Ch} extended Nitsche's classification when the number of singular points is finite and some additional conditions on the angles of the vertices given by these singular points. 

Recently, A. Fraser and R. Schoen in \cite{FS} showed that any two-dimensional free boundary immersion with parallel mean curvature in the Euclidean ball of $\mathbb{R}^{n}$ must be totally umbilical. They proved that the complex quartic differential obtained by squaring the Hopf differential of the surface is holomorphic.\\

The above classifications results show that the existence of a holomorphic quadratic differential is a useful tool to classify $H-$ surfaces of $\mathbb{R}^{3}$. Indeed, the underlying idea for the construction of the Hopf differential relies in the Codazzi pair that can be defined on the surface, for example, on a constant mean curvature surface in $\mathbb{R}^{3}$, the Codazzi pair is obtained by the first and the second fundamental form. Under some geometrical conditions, a Codazzi pair gives rise to holomorphic quadratic differential on the surface.

For surfaces in the homogeneous spaces with 4-dimensional isometry group $\hmf$, U. Abresch and H. Rosenberg in \cite{Rosen} showed the existence of a holomorphic quadratic differential, the Abresch-Rosenberg differential, on any $H-$ surface. Hence, they classified the topological spheres with constant mean curvature as the rotationally symmetric ones.

In the same spirit of J. Nitsche, I. Fernández and M. Do Carmo used the Abresch-Rosenberg differential to classify $H-$disks in the product spaces $\mathbb{S}^{2} \times \mathbb{R}$ and $\mathbb{H}^{2} \times \mathbb{R}$; they proved that the $H-$ disk that meets an Abresch-Rosenberg surface, i.e, those whose Abresch-Rosenberg differential along its boundary at a constant angle is part of an Abresch-Rosenberg surface under certain geometrical conditions. Similar results have been obtained by M. P. Cavalcante and J. Lira in \cite{Cal}, they proved that a $H-$disk in the product spaces that meets a slice at a constant angle is a cap.\\ 

The aim of this paper is to classify $H-$ disks in any homogeneous space $\hmf$. Recently, in \cite{ER4} the authors defined a Codazzi pair $(I,II_{AR})$ such that the $(2,0)$- part of $II_{AR}$ is the Abresch-Rosenberg differential for a $H-$ surface in $\hmf$. Then, we use the mentioned Codazzi pair to establish a Joachimstahl's Type Theorem for the intersection of $H-$ surfaces, in this way, we can use the Abresch-Rosenberg differential to  classify $H-$ disks in $\hmf$.

As a consequence, we will show that our results generalize the previous classification of $H-$disks in the product spaces $\mathbb{S}^{2} \times \mathbb{R}$ and $\mathbb{H}^{2} \times \mathbb{R}$ by I. Fernández and M. Do Carmo and we will get a classification for $H-$ disks in  $\hmf$, when $\t \neq 0$.\\ 

The content of this paper is organized as follows: 

In section 2, we set up the notation that will be used along this paper and we review standard facts about $H-$surfaces in $\hmf$, as well as, the abstract theory of the Codazzi pairs in Riemannian surfaces. Finally, we recall the definition of the Abresch-Rosenberg shape operator for a $H-$ surface in $\hmf$ and its properties.\\

In section 3, we define the lines of curvature respect to the Abresch-Rosenberg shape operator  and we show the Key Lemma of this work. In fact, the Key Lemma is a Joachimstalh's Type Theorem for the intersection of $H-$ surfaces in $\hmf$. Later, we prove that natural geometric configurations about the intersection of the $H-$ surfaces imply the conditions of the Key Lemma.\\

Finally, in section 4, we apply the Key Lemma to classify $H-$ disks in $\hmf$ that meet an Abresch-Rosenberg surface along its boundary at a constant angle. First, we classify these types surfaces with smooth boundary and assuming certain geometrical conditions. After, we extend the previous classification to piece-wise differentiable boundary case assuming that the number of singular points is finite and additional suppositions on the angles of the vertices by using a general result about Codazzi pairs in \cite{ER3}.

{\bf Acknowledgements}. I am grateful to professor José María Espinar  for his constant encouragement and support to prepare this paper. Also, I would like to thank the warm welcome by the people at the Instituto de Matematica Pura e Aplicada (IMPA) where this work was initiated and the Universidad Federal Fluminense (UFF) for its hospitality, where this paper was concluded.

\section{Preliminaries}
In this preliminary section, we summarize  the general theory of $H-$ surfaces in homogeneous Riemannian manifolds and the abstract theory of Codazzi pairs in Riemannian surfaces, for this we follow the references \cite{AEG3,D,ER,ER2,Scott}. 

\subsection{Homogeneous Riemanniann Manifolds $\mathbb{E}(\kappa ,\tau)$.}

The simply connected homogeneous Riemannian manifolds whose isometry group have dimension 4 are the three dimensional Riemannian manifolds  that fiber over a simply connected two dimensional manifolds space form $\mathbb{M}^{2}(\k)$ and its fibers are the trajectories of a unitary Killing vector field $\xi$. According to standard notation, we embrace these spaces as $\mathbb{E}(\kappa ,\tau)$, where $\kappa$ and $\tau$ are constants so that  $\kappa - 4\tau^{2} \neq 0$.  

In fact, the homogeneous manifolds $\hmf$ can be classified  depending of numbers $\kappa$ and $\tau$ (\cite{Scott}); if $\tau = 0$,  $\hmf$ is the product space $\mer$, where $\me = \s ^2(\kappa)$ if $\kappa >0$ ($\s ^2(\kappa)$ the sphere of constant curvature $\kappa$), or $\me  =  \h ^2(\kappa)$ if $\kappa < 0$ ($\h^2(\kappa)$ the hyperbolic plane of constant curvature $\kappa$).  If $\tau \neq 0$, $\hmf $ is either a Berger sphere if $\kappa > 0$, or a Heisenberg space if $\kappa = 0$, or the universal cover of ${\rm PSL}(2,\r)$ if $\kappa < 0$. Henceforth we will suppose $\kappa$ is plus or minus one or zero. 

As we said, The homogeneous space $\hmf $ is a Riemannian submersion $\pi: \hmf \rightarrow \M^{2}(\kappa )$ over a simply connected surface of constant sectional curvature $\kappa$. The fibers are the inverse image of a point at $\M^{2}(\kappa)$ by $\pi$. The fibers are the trajectories of a unitary Killing field $\xi$, called the vertical vector field.

Denote by $\overline{\nabla}$ the Levi-Civita connection of $\hmf$, then for all $X \in \mx (\hmf )$, the following equation holds \cite{Scott}:
$$\overline{\nabla}_{X}\xi = \t  X \wedge \xi ,$$where $\t$ is the bundle curvature. 

\subsection{Immersed surfaces in $\hmf$.}

Let $\Sigma \subset \hmf$ be an oriented immersed connected surface. We endow $\Sigma$ with the induced metric of $\hmf$, the first fundamental form, which we  denote by $\meta{}{}$. Denote by $\nabla$, $R$ and $A$ the Levi-Civita connection, the Riemann curvature tensor and the shape operator of $\Sigma$ respectively. So, $$AX = -\overline{\nabla}_{X}N  \text{ for all }X \in \mx(\Sigma), $$where $N$ is the unit normal vector field along the surface $\Sigma$. Then $II(X,Y) = \langle AX,Y \rangle$ is the  second fundamental Form of $\Sigma$. 

Moreover, denote by $J$ the oriented rotation of angle $\frac{\pi}{2}$ on $T\Sigma$, defined by the formula 
$$JX = N \wedge X \text{ for all } X \in \mx(\Sigma) ,$$ 

where $\wedge$ denotes the wedge product of vector fields. Set $\nu = \langle N,\xi\rangle$ and $\mt = \xi - \nu N$, that is, $\nu$ is the normal component of the vertical vector field $\xi$, called the angle function, and $\mt$ is the tangent component of the vertical vector field $\xi$.

In terms of a local conformal parameter $z$, the first fundamental form $I = \la,\ra$ and the second fundamental form are given by 

\begin{eqnarray}
I & = & 2\lambda |dz|^{2} \\
II & = & Q dz^{2} + 2 \lambda H |dz|^{2} + \overline{Q}d\overline{z}^{2},
\end{eqnarray}where $Q dz^{2} = - \la \overline{\nabla}_{\partial_{z}}N,\partial_{z} \ra$ is the usual Hopf differential of $\Sigma$. Hence, in this conformal coordinate, we have the following:

\begin{lema}[\cite{ER,ER2}]
\label{LemaEQ}
Given an immersed $H-$surface $\Sigma \subset \hmf$, the following equations are satisfied: 
\begin{eqnarray}
K & = & K_{e} + \t^{2} + (\kappa - 4\t^{2})\nu^{2} \label{GaussZ}\\
Q_{\overline{z}} & = & \lambda(\kappa - 4\t^{2})\nu \nt \label{CodazziZ}\\
\nt_{z} & = & \frac{\lambda_{z}}{\lambda}\nt + Q\nu \label{NablaTZ}\\
\nt_{\overline{z}} & = & \lambda(H + i \t) \nu \label{NablaTZb}\\
\nu_{z} & = & -(H - i\t)\nt -\frac{Q}{\lambda}\overline{\nt}  \label{DerNuZ} \\
|\nt |^{2} & = & \frac{1}{2} \lambda(1-\nu^{2}) \label{normZ},
\end{eqnarray}where $\nt = \la \mt, \partial_{z} \ra$, $\overline{\nt} = \la \mt, \partial_{\overline{z}} \ra$, $K_{e}$ is the extrinsic curvature and $K$ is the Gaussian curvature of $\Sigma$.
\end{lema}

For an immersed $H-$surface $\Sigma \subset \hmf$, we can define a global quadratic differential.

\begin{definition}[\cite{Rosen,Rosen1}]\label{DefAR}
Given a local conformal parameter $z$ for $I$, the Abresch-Rosenberg differential is defined by:
$$\mathcal{Q}^{AR} = Q^{AR}dz^{2} = (2(H+i \t)Q - (\k - 4\t^{2})\nt^{2})dz^{2}.$$ 

Note that $\mathcal{Q}^{AR}$ do not depend on the conformal parameter $z$, hence $\mathcal{Q}^{AR}$ is globally defined on $\Sigma$.

\end{definition}

Therefore using the definition of the Abresch-Rosenberg differential and Lema \ref{LemaEQ}, we obtain

\begin{theorem}[\cite{Rosen,Rosen1}]
\label{thQH}
Let $\Sigma$ be a $H$- surface in $\hmf$, then the Abresch-Rosenberg differential $\mathcal{Q}^{AR}$ is holomorphic for the conformal structure induced by the first fundamental form $I$.
\end{theorem}

In the following we recall the classification of the complete $H-$surfaces in $\hmf$ whose the Abresch-Rosenberg differential vanish. 

\begin{theorem}[\cite{Rosen,Rosen1,ER2}]
\label{QVa}
Let $\Sigma \subset \hmf$ be a complete $H-$ surface whose Abresch-Rosenberg differential vanishes. Then $\Sigma$ is invariant by a one  parameter group of isometries of $\hmf$. Moreover, $\Sigma$ is either a slice in $\mathbb{S}^{2} \times \R$ or $\mathbb{H}^{2} \times \R$ if $H = 0 = \t$ and in the case $H^{2}+\t^{2} \neq 0$ the Gauss curvature $K$ of these examples satisfies:

\begin{itemize}
\item If either $4(H^{2} + \t^{2}) > \k-4\t^{2}$ when $\k-4\t^{2}>0$ or $H^{2} + \t^{2}> -(\k-4\t^{2})$ when $\k-4\t^{2}<0$, then $K>0$, i.e, $\Sigma$ is a rotationally invariant sphere. In particular, $4H^{2} + \k >0$ .

\item If $4H^{2} + \k = 0$ and $\nu = 0$, then $K= 0$, i.e $\Sigma$ is either a vertical plane in ${\rm Nil}_{3}$ or a vertical cylinder over a horocycle in $\mathbb{H}^{2} \times \R$ or $\widetilde{{\rm PSL}(2,\r)}$.

\item There exists a point with negative Gauss curvature in the rest of examples.
\end{itemize}
\end{theorem}

Then the above theorem motivates the following definition

\begin{definition}\label{Def:ARSurface}
A complete $H-$ surface $\Sigma$ immersed in $\hmf$ is an  Abresch-Rosenberg surface if the Abresch-Rosenberg differential $\mathcal{Q}^{AR}$ vanishes on $\Sigma$. In particular by Theorem \ref{QVa}, the surface $\Sigma$ must be invariant by a one parameter group of isometries of $\hmf$.
\end{definition}

In Figure 1, we show the meridians of the  Abresch-Rosenberg surfaces in the product spaces $\mathbb{S}^{2} (\kappa) \times \R$ and $\mathbb{H}^{2}(\kappa) \times \mathbb{R}$. In \cite{Rosen}, the authors gave a complete description of these surfaces.

\begin{figure}[!ht]
    \subfloat[Meridians of CMC spheres $S^{2}_{H}$ in \newline $\mathbb{S}^{2}(\kappa) \times \mathbb{R}$ ]{%
      \begin{overpic}[width=0.50\textwidth]{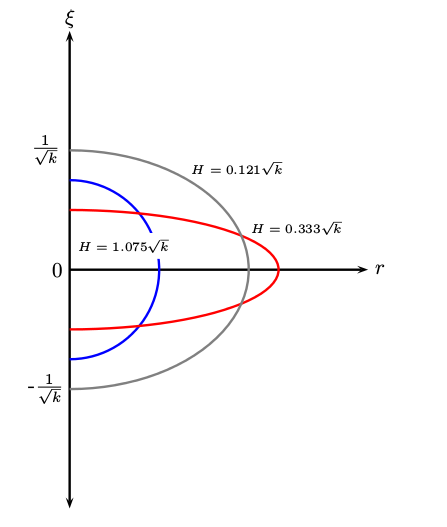}
      
      \end{overpic}
    }
    \subfloat[Meridians of the CMC surfaces $C^{2}_{H}$ of catenoidal type in $\mathbb{H}^{2}(\kappa) \times \mathbb{R}$]{%
      \begin{overpic}[width=0.50\textwidth]{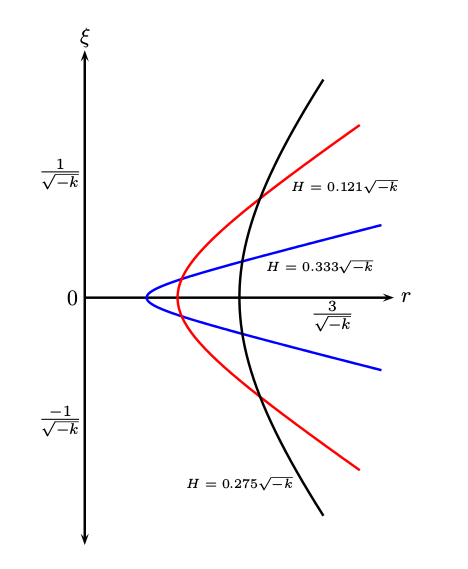}
      
      \end{overpic}
    }
    
\subfloat[Meridians of CMC
spheres $S^{2}_{H} $ and \newline  disk-like cmc surfaces $D_H^{2}$  in $\mathbb{H}^{2}(\kappa) \times \mathbb{R}$]{%
      \begin{overpic}[width=0.50\textwidth]{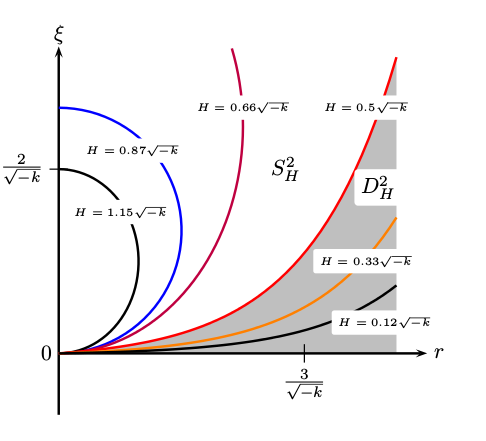}
      
      \end{overpic}
    }
    \subfloat[ Parabolic CMC surfaces $P_{H}^{2}$ come as limits in $\mathbb{H}^{2}(\kappa) \times \mathbb{R}$]{%
      \begin{overpic}[width=0.50\textwidth]{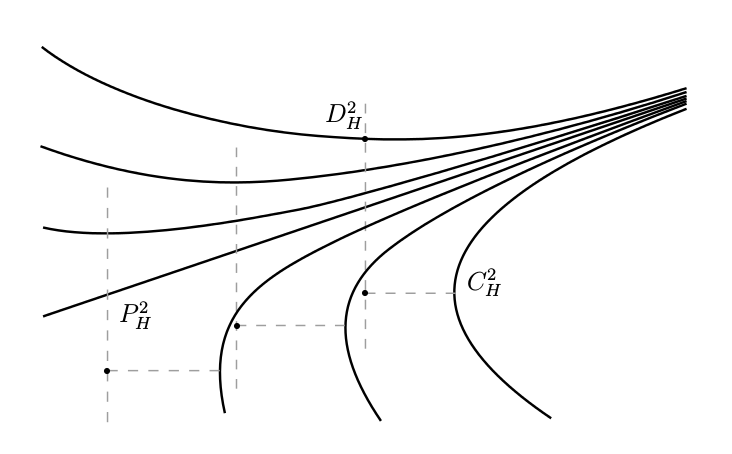}
      
      \end{overpic}    
    }
    \caption{ Meridians of CMC surfaces in the product spaces $\mathbb{M}^{2}(\kappa) \times \mathbb{R}$ whose the Abresch-Rosenberg differential vanishes}
  \end{figure}

\subsection{Codazzi Pairs on Surfaces.}

One important tool in this paper is Codazzi pairs. We follow \cite{AEG3} and references therein. We shall denote by $\Sigma$ an orientable (and oriented) smooth Riemannian surface. Otherwise we would work with its oriented two-sheeted covering.

\begin{definition}
A fundamental pair on $\Sigma$ is a pair of real quadratic forms
$(I,II)$ on $\Sigma$, where $I$ is a Riemannian metric.
\end{definition}

Associated with a fundamental pair $(I,II)$ we define the shape operator $S$ of the pair as:
\begin{equation}\label{ii}
 II(X,Y)=I(S(X),Y) \text{ for any } X,Y \in \mx (\Sigma).
\end{equation}

Conversely, it is clear from (\ref{ii})  that the quadratic form $II$ is totally
determined by $I$ and $S$. In other words, to give a fundamental pair on $\Sigma$ is equivalent to give a Riemannian metric on $\Sigma$ together with a self-adjoint endomorphism $S$.

We define the  mean curvature $H(I,II)$, the extrinsic curvature $K(I,II)$ and the  principal curvatures of the fundamental pair $(I,II)$ as one half of the trace, the determinant and the eigenvalues of the endomorphism $S$, respectively.

Now, consider $\Sigma$ as a Riemann surface with respect to the metric $I$ and take a local conformal parameter $z$, then we can write
\begin{equation}\label{parholomorfo}
\begin{split} 
I &=2\lambda\,|dz|^2, \\
II &=Q(I,II)\,dz^2+2\lambda\,H\,|dz|^2+\overline{Q}(I,II)\,d\zb^2.
\end{split}
\end{equation}

The quadratic form $Q(I,II)\,dz^2$, which does not depend on the chosen parameter $z$, is known as the  Hopf differential of the pair $(I,II)$. A specially interesting case happens when the fundamental pair satisfies the Codazzi equation, that is,

\begin{definition}
We say that a fundamental pair $(I,II)$, with shape operator $S$, is a
Codazzi pair if
\begin{equation}\label{ecuacioncodazzi}
\nabla_XSY-\nabla_YSX-S[X,Y]=0,\qquad X,Y\in\mx(\Sigma),
\end{equation}
where $\nabla$ stands for the Levi-Civita connection associated with the Riemannian metric $I$ and $\mx (\Sigma)$ is the set of smooth vector fields on $\Sigma$.
\end{definition}

Let us also observe that, by equations (\ref{parholomorfo}) and (\ref{ecuacioncodazzi}), we obtain that the fundamental pair $(I,II)$ is a Codazzi pair if and only if
$$
Q(I,II)_{\zb}=\lambda\,H(I,II)_z.
$$

Thus, one has the T.K Milnor Theorem:
\begin{lema}[\cite{Mi}]\label{l0.1}
Let $(I,II)$ be a fundamental pair. Then, any two of the conditions {\rm (i)}, {\rm (ii)}, {\rm (iii)} imply the third:
\begin{itemize}
\item[{\rm (i)}] $(I,II)$ is a Codazzi pair.
\item[{\rm (ii)}] $H(I,II)$ is constant.
\item[{\rm (iii)}] The Hopf differential $Q(I,II)\,dz^2$ of the pair $(I,II)$ is holomorphic.
\end{itemize}
\end{lema}

\subsection{The Abresch-Rosenberg shape operator}

Now, we  recall the definition of the Codazzi pair on any $H-$surface in $\hmf$ such that the Abresch-Rosenberg differential appears as its Hopf differential.  First, we study the case $\t = 0$. After, we study the case $\t \neq 0$.

\subsection{$H-$surfaces in $\mer$.}

Consider a complete immersed $H-$surface $\Sigma \subset \mathbb{M}^{2}(\kappa) \times \mathbb{R}$. According to the notation introduced above, we define the self-adjoint endomorphism $S$ along $\Sigma$ as
\begin{equation}
S_{AR}X = 2H \, AX - \kappa \langle X,\mt \rangle \mt + \frac{\kappa}{2}\abs{\mt}^{2}X - 2H^{2}X,
\label{Oper}
\end{equation}where $X \in \mathfrak{X}(\Sigma)$. Now consider the quadratic form $II_{AR}$ associated to $S_{AR}$  given by \eqref{Oper}. In \cite{AEG3}, it was shown that $(I,II_{AR})$ is a Codazzi pair on $\Sigma$ when $H$ is constant. Moreover, it is traceless, i.e., ${\rm tr}(S_{AR}) = 0= H(I,II_{AR})$, and the Hopf differential associated to $(I,II_{AR})$ is the Abresch-Rosenberg differential $\mathcal Q^{AR}$ in $\mathbb{M}^{2}(\kappa) \times \R$.

\subsection{$H-$surfaces in $\hmf$ with $\t \neq 0$.}

Let $\Sigma$ be a $H-$surface in $\hmf$, $\t \neq 0$. In this case we have that $H^{2} + \t^{2} > 0$. According to the notation introduced above, we define the self-adjoint endomorphism $S$ along $\Sigma$ as 
\begin{equation}\label{ARTraceless}
S_{AR}X = A(X) - \alpha \langle \mt_{\theta}, X \rangle \mt_{\theta} + \frac{\alpha \abs{\mt}^{2}}{2}X - HX,
\end{equation} where 
\begin{itemize}
\item $\alpha = \dfrac{\k - 4\t^{2}}{2 \sqrt{H^{2} + \t^{2}}}$,
\item $e^{2i\te} = \frac{H - i\t}{\sqrt{H^{2} + \t^{2}}}$ and 
\item $\mt_{\te} = \cos\te \mt + \sin\te J\mt$.
\end{itemize}

Now, consider the quadratic form $II_{AR}$ associated to $S_{AR}$  given by \eqref{ARTraceless}. In \cite{ER4}, it was shown that $(I,II_{AR})$ is a Codazzi pair on $\Sigma$ when $H$ is constant. Moreover, it is traceless and the Hopf differential associated to $(I,II_{AR})$ is the Abresch-Rosenberg differential $\mathcal Q^{AR}$ in $\mathbb{M}^{2}(\kappa) \times \R$ up to the constant $H+i\tau$.

\section{Abresch-Rosenberg lines of curvature.}

We begin this section by defining the concept of line of curvature with respect to the Abresch-Rosenberg shape operator of a $H-$surface $\Sigma$ in $\hmf$ and we characterize these curves in terms of the Abresch-Rosenberg differential. 

\begin{definition}
\label{Def5}
Let $\Sigma$ be a $H-$surface in $\hmf$ and $\Gamma$ a regular curve parametrized by $\gamma:(-\epsilon,\epsilon) \rightarrow \Sigma$. We say that $\Gamma = \gamma (-\epsilon, \epsilon )$ is a line of curvature with respect to the Abresch-Rosenberg shape operator $S_{AR}$ if  there exists a smooth function $\lambda: (-\epsilon,\epsilon) \rightarrow \mathbb{R}$ such that $S_{AR}(\gamma'(t)) = \lambda(t)\gamma'(t)$. In such case, we call $\Gamma$ an  Abresch-Rosenberg line of curvature, in short, an AR-line of curvature.
\end{definition} 

Definition \ref{Def5} says that the tangent vector of $\gamma$ is an eigenvector of $S_{AR}$ along $\gamma$. So, the definition is nothing but the natural extension of a line of curvature to the case of the Abresch-Rosenberg shape operator $S_{AR}$ for a $H-$surface in $\hmf$. In analogy with the situation in $\R^{3}$ (see \cite{Ch}), there exists a link  between lines of curvatures with respect to $S_{AR}$ and the Abresch-Rosenberg differential $Q^{AR}dz^{2}$.

\begin{prop}
\label{A-R}
Let $\Sigma$ be a $H-$surface in $\hmf$. Then,  $\gamma:(-\epsilon,\epsilon) \rightarrow \Sigma$ is a line of curvature for $S_{AR}$ if, and only if, the imaginary part of $Q^{AR}dz^{2}$ vanishes along $\gamma$.
\end{prop}

\begin{proof}
Let $z=u+iv$ be a local conformal parameter of $\Sigma$ for the first fundamental form $I$ and set
\begin{equation*}
\begin{split}
II_{AR}(\partial_{u},\partial_{u}) & = I(S_{AR}(\partial_{u}),\partial_{u}) = L,\\
II_{AR}(\partial_{u},\partial_{v}) & = I(S_{AR}(\partial_{u}),\partial_{v}) = M,\\
II_{AR}(\partial_{v},\partial_{v}) & = I(S_{AR}(\partial_{v}),\partial_{v}) = N.
\end{split}
\end{equation*} 

The curve $\gamma(t)= (u(t),v(t))$ is a line of curvature with respect to $S_{AR}$ if, and only if, $S_{AR}(\gamma'(t)) = \lambda(t) \gamma'(t)$. In the local coordinates $(u,v)$ this means
\begin{equation}
\label{Matrix}
\begin{bmatrix}
    L(\gamma(t)) & M(\gamma(t)) \\
    M(\gamma(t)) & N(\gamma(t))
\end{bmatrix} 
\begin{bmatrix}
u'(t)\\
v'(t)
\end{bmatrix}
=
\lambda(t)
\begin{bmatrix}
u'(t)\\
v'(t)
\end{bmatrix} .
\end{equation}

Hence, from \eqref{Matrix}, we get the following linear system
\begin{equation}
\label{sys}
\begin{split}
L(\gamma(t))u'(t) + M(\gamma(t))v'(t) & = \lambda(t)u'(t) ,\\
M(\gamma(t))u'(t) + N(\gamma(t))v'(t) & = \lambda(t)v'(t).\\
\end{split}
\end{equation}

On the one hand, from \eqref{sys}, we obtain
\begin{equation}
\label{Rel}
M(\gamma(t))(v'(t))^{2} + (L(\gamma(t))-N(\gamma(t))) u'(t)v'(t) - M(\gamma(t))(u'(t))^{2} = 0 .
\end{equation} 

On other hand, from the definition of the Abresch-Rosenberg differential
\begin{equation}
\label{ABLC}
Q^{AR}(\gamma(t))dz(\gamma(t))^{2} = (L(\gamma(t))-N(\gamma(t))-2iM(\gamma(t)))dz(\gamma(t))^{2},
\end{equation}then, a straightforward computation shows that the imaginary part is given by  
\begin{equation}\label{Impart}
\begin{split}
{\rm Im}(Q^{AR}(\gamma(t))dz(\gamma(t))^{2})&= M(\gamma(t))(v'(t))^{2} + (L(\gamma(t))-N(\gamma(t))) u'(t)v'(t) \\
 & \qquad - M(\gamma(t))(u'(t))^{2},
\end{split}
\end{equation}hence, \eqref{Impart} is nothing but the left hand side of \eqref{Rel}. This shows that the imaginary part of \eqref{ABLC} vanishes when $\gamma(t)$ is line of curvature with respect to $S_{AR}$.
 
Reciprocally, assume that the imaginary part of $Q^{AR}(\gamma(t))dz(\gamma(t))^{2}$ is zero, this condition is given by \eqref{Impart}. Then, for all $t \in (-\epsilon ,\epsilon) $ so that $u'(t) \neq 0$ and $v'(t) \neq 0$ we have
\begin{equation} 
\label{eqlam}
\frac{M(\gamma(t))v'(t) + L(\gamma(t))u'(t)}{u'(t)} = \frac{N(\gamma(t))v'(t) + M(\gamma(t)u'(t))}{v'(t)}.
\end{equation}

Now, define the function $\lambda: (-\epsilon ,\epsilon) \rightarrow \mathbb{R}$ as follows
\begin{equation*}
\lambda(t) =
\begin{cases}
 \frac{M(\gamma(t))v'(t) + L(\gamma(t))u'(t)}{u'(t)}& \text{for $t$ such that $u'(t) \neq 0$ and $v'(t) \neq 0$},\\
L(\gamma(t)) & \text{for $t$ such that $u'(t) \neq 0$ and $v'(t) = 0$},\\
N(\gamma(t)) & \text{for $t$ such that $u'(t) = 0$ and $v'(t) \neq 0$}.
\end{cases}
\end{equation*}

Therefore, \eqref{sys} and \eqref{eqlam} imply $S_{AR}(\gamma'(t)) = \lambda(t) \gamma'(t)$ for all $t \in (-\epsilon ,\epsilon)$, this shows $\gamma'(t)$ is line of curvature respect to $S_{AR}$.
\end{proof}

Next, we will establish a Joachimstahl's Type Theorem for the intersection of $H-$surfaces in $\hmf$. This is a key step in this work. 

\begin{lema}[\bf{Key Lemma}] \label{key}
Let $\Sigma_{i}\subset \hmf$  $i=1,2$ be $H_{i}-$surfaces so that $\Sigma _{1}\cap \Sigma _{2} \neq \emptyset$. Let $\Gamma \subset \Sigma _{1}\cap \Sigma _{2}$ be a regular curve of transversal intersection. Assume that along $\Gamma$ it holds
\begin{enumerate}
\item[a)] $\meta{N_{1}}{N_{2}} $ is constant and
\item[b)] $\sqrt{H^{2}_{1}+\tau ^{2}}\la \mt ^{2}_{\theta_{2}},N_{1} \ra \la J_{2} \mt ^{2}_{\theta_{2}},N_{1} \ra = \sqrt{H^{2}_{2}+\tau ^{2}}\la \mt ^{1}_{\theta_{1}},N_{2} \ra \la J_{1} \mt ^{1}_{\theta_{1}},N_{2} \ra $,
\end{enumerate}where $e^{2i\te _i} = \frac{H _i - i\t}{\sqrt{H_i^{2} + \t^{2}}}$, $\mt ^{i}_{\theta_{i}} = \cos\theta_{i}\mt_{i} + \sin\theta_{i}J\mt_{i}$ and $J\mt ^{i}_{\theta_{i}} = N_{i}\wedge \mt ^{i}_{\theta_{i}}$ for $i=1,2$. Then, $\Gamma $ is an AR-line of curvature for $\Sigma _1$ if, and only if, $\Gamma$ is an AR-line of curvature for $\Sigma _2$.
\end{lema}

\begin{proof}
We assume that $\Gamma$ is an AR-line of curvature for $\Sigma _{2}$, the other case is completely analogous. First, since $\la N_{1}(\gamma(t)), N_{2}(\gamma(t)) \ra$ is constant along $\Gamma = \gamma (-\epsilon, \epsilon)$, where $\gamma $ is parametrized by arc-length, then 
\begin{equation}
\label{SOSF3}
\la A_{1}(\gamma'(t)),N_{2}(\gamma(t)) \ra + \la N_{1}(\gamma(t)), A_{2}(\gamma'(t)) \ra = 0 ,
\end{equation}where $A_{1}$ and $A_{2}$ are the shape operators  of the second fundamental forms of $\Sigma _{1}$ and $\Sigma _{2} $ respectively. Now, relating $A_{1}$ and $A_{2}$ with $S_{AR}^{1}$ and $S_{AR}^{2}$ respectively and using that $\gamma(t)$ is a line of curvature for $S_{AR}^{2}$, we can rewrite \eqref{SOSF3} as: 
\begin{equation}
\label{eqt}
-\alpha_{1} \la \mt^{1}_{\theta_{1}} , \gamma'(t) \ra \la \mt^{1}_{\theta_{1}},N_{2}\ra - \la  S_{AR}^{1}(\gamma'(t)), N_{2}\ra - \alpha_{2} \la \mt^{2}_{\theta_{2}}, \gamma'(t) \ra \la \mt^{2}_{\theta_{2}}, N_{1} \rangle = 0,
\end{equation}where $\alpha_{i} = \frac{\k-4\t^{2}}{2 \sqrt{H^{2}_{i}+\t^{2}}}$ and $\mt ^{i}_{\theta_{i}} = \cos\theta_{i}\mt_{i} + \sin\theta_{i}J\mt_{i}$, $i=1,2$.

We can orient $\gamma$ so that $ (1-d^{2}) ^{-1}N_{1} \wedge N_{2} =  \gamma'(t)$, where $d$ is the contact constant angle between $\Sigma _{1} $ and $\Sigma _{2}$.

Since the intersection is transversal, $\{ N_{1}, N_{2}, \gamma'(t) \}$ is an oriented basis of $T_{\gamma(t)}\hmf$ for each $t$ were the intersection is transversal. Then, the following equations hold:
\begin{equation}
\label{eqt31}
\begin{split}
\la \mt^{1}_{\theta_{1}}, \gamma'(t) \ra & = (1-d^{2})^{-1} \la \mt^{1}_{\theta_{1}}, N_{1}\wedge N_{2} \ra = - (1-d^{2})^{-1}\la J_{1} \mt ^{1}_{\theta_{1}},N_{2} \ra , \\
\la \mt^{2}_{\theta_{2}}, \gamma'(t) \ra & = (1-d^{2})^{-1} \la \mt^{2}_{\theta_{2}}, N_{1}\wedge N_{2} \ra =  (1-d^{2})^{-1}\la J_{2} \mt ^{2}_{\theta_{2}},N_{1} \ra ,
\end{split}
\end{equation}where $J_{i} \mt ^{i}_{\theta_{i}} = N_{i} \wedge \mt ^{i}_{\theta _{i}}$ for $i=1,2$. Therefore, \eqref{eqt} and \eqref{eqt3} imply that 
\begin{equation*}
\begin{split}
\la S_{AR}^{1}(\gamma'(t)), N_{2}\ra &=  \frac{(1-d^{2})^{-1}(\kappa -4\tau ^{2})}{\sqrt{H^{2}_{1}+\tau ^{2}}}\la \mt ^{1}_{\theta_{1}},N_{2} \ra \la J_{1} \mt ^{1}_{\theta_{1}},N_{2} \ra  \\
 & \qquad -\frac{(1-d^{2})^{-1}(\kappa -4\tau ^{2})}{\sqrt{H^{2}_{2}+\tau ^{2}}}\la \mt ^{2}_{\theta_{2}},N_{1} \ra \la J_{2} \mt ^{2}_{\theta_{2}},N_{1} \ra \\
  & = 0,
\end{split}
\end{equation*}where we have used item $b)$. Thus, $\la S_{AR}^{1}(\gamma'(t)), N_{2}\ra = 0 $ along $\Gamma$. Therefore, since $\la S_{AR}^{1}(\gamma'(t)), N_{1}\ra = 0 $ along $\Gamma$, we obtain that $S_{AR}^{1}(\gamma'(t)) = \lambda (t) \gamma'(t)$, that is, $\Gamma$ is an AR-line of curvature for $\Sigma _1$.
\end{proof}

Observe that we can see item $b)$ above geometrically as follows. Let $X$ be a unitary vector field along $\Sigma$, not necessarily tangential. Then, $\set{\mt _\theta , J \mt _\theta}$ is an orthogonal frame along $\Sigma$ away from the points where $\abs{\mt} =0$. Let $\beta $ be the (oriented) angle between $\mt $ and $X$, that is, 
$$ \meta{\mt }{X} = \abs{\mt} \cos \beta ,$$and hence, 
$$ 2 \meta{\mt _\theta }{X}\meta{J \mt _\theta}{X} = \abs{\mt }^2 \sin (2(\beta - \theta)) .$$

So, coming back to the situation on the Key Lemma, let $\omega _{ij}$ denote the (oriented) angle between $\mt ^i$ and $N_j$, for $i,j =1,2$ and $i \neq j$. Hence, item $b)$ can be re-written as 
$$\frac{\abs{\mt _1}^2}{\sqrt{H^{2}_{1}+\tau ^{2}}}\sin (2(\omega _{12}-\theta _1))= \frac{\abs{\mt _2}^2}{\sqrt{H^{2}_{2}+\tau ^{2}}}\sin (2(\omega _{21}-\theta_2)) . $$

The Key Lemma gives us general conditions for $\Gamma$ being an AR-line of curvature of both surfaces $\Sigma _1$ and $\Sigma _2$. Nevertheless, we will see that certain geometric configurations imply the conditions on the Key Lemma.\\

A curve $\Gamma$ on a surface $\Sigma$ in $\mer$ is  horizontal if $\Gamma$ is contained in a horizontal slice $\mathbb{M}^{2}(\k) \times \{ \xi_{0} \}$, for some $\xi_{0} \in \mathbb{R}$. On the other hand, the curve $\Gamma$ in $\Sigma$ is said to be vertical if it is an integral curve of the vector field  $\mt$. 

\begin{Corollary}
\label{CoroDCF}
Let $\Sigma_{1}$ and $\Sigma_{2}$ two constant mean curvature surfaces in $\mer$ with mean curvatures $H_{1}$ and $H_{2}$, normal vectors $N_{1}$ and $N_{2}$ and angle functions $\nu_{1}$ and $\nu_{2}$, respectively. Let  $\Gamma \subset \Sigma _{1}\cap \Sigma _{2}$ be a regular curve parametrized by $\gamma$. Suppose that  $\Sigma_{1}$ and $\Sigma_{2}$ intersects transversally along $\Gamma$ at a constant angle and $\Gamma$ is a AR-line of curvature for $\Sigma_{1}$. Assume along $\Gamma$ one of the following conditions holds:

\begin{enumerate}
\item $\Gamma$ is an horizontal curve of $\Sigma_{1}$.
\item $\Gamma$ is a vertical curve of $\Sigma_{1}$ and  $\Sigma_{2}$.
\item If $H_{1}=H_{2} \neq 0$ , the angle function $\nu_{1}$ is opposite to the angle function $\nu_{2}$. 
\end{enumerate}

Then $\Gamma$ is an AR-line of curvature for $\Sigma_{2}$. 
\end{Corollary}

\begin{proof}
In the first case, assume that $\Gamma = \gamma(-\epsilon,\epsilon)$ where $\gamma$ is parametrized by arc-length. In $\mer$, we have   $\mt^{1}_{\theta_1} = \mt_{1}$ and $\mt^{2}_{\theta_{2}} = \mt_{2}$. Since $\gamma$ is horizontal, we obtain

\begin{equation}
\label{eqt33}
\begin{split}
 (1-d^{2})\la J_{1} \mt ^{1}_{\theta_{1}},N_{2} \ra &= - \la \mt^{1}, \gamma'(t) \ra = - \la \xi, \gamma'(t) \ra = 0,\\
(1-d^{2})\la J_{2} \mt ^{2}_{\theta_{2}},N_{1} \ra &= \la \mt^{2},\gamma'(t) \ra = \la \xi,\gamma'(t) \ra = 0.
\end{split}
\end{equation}

From \eqref{eqt33} and Lemma \ref{key}, then $\gamma(t)$ is AR-line of curvature of $\Sigma_{2}$.\\

In the second case, $\mt_{2}=\mt_{1}= \gamma'(t)$ for each $t$, hence $\la \mt_{1},N_{2} \ra = \la \mt_{2},N_{1} \ra = 0$ so the hypothesis of Lemma \ref{key} holds clearly and then $\gamma$ is a AR-line of curvature of $\Sigma_{2}$.\\

In the third case, suppose $H_{1}=H_{2}=H$ and $\nu_{1}=-\nu_{2}$, therefore
\begin{equation}
\label{Epro}
\begin{split}
H \la\mt_{2},N_{1} \ra \la \mt_{2},\gamma'(t) \ra & = H \la \mt_{2},N_{1}\ra \la \xi,\gamma'(t) \ra  \\ 
& = H (\nu_{2}-\nu_{1}d) \la \mt_{1}+\nu_{1}N_{1},\gamma'(t) \ra \\ 
& = -H(\nu_{1}- \nu_{2}d) \la \mt_{1},\gamma'(t)\ra \\
& = -H \la \mt_{1},N_{2} \ra \la \mt_{1},\gamma'(t) \ra.
\end{split}
\end{equation}

Hence, from  \eqref{eqt31} and \eqref{Epro}, we can see again that the hypothesis of Lemma \ref{key} hold. Then, in any case, $\Gamma$ is a AR-line of curvature of $\Sigma_{2}$.
\end{proof}

Next, we give certain geometric configurations for $H-$ surfaces in $\hmf$, $\t \neq 0$, that imply the Key Lemma.

\begin{Corollary}
\label{CorodCF}
Let $\Sigma_{1}$ and $\Sigma_{2}$ two $H-$ surfaces in $\hmf$, $\t \neq 0$, with normal vectors $N_{1}$ and $N_{2}$ and angle functions $\nu_{1}$ and $\nu_{2}$, respectively. Let $\Gamma \subset \Sigma _{1}\cap \Sigma _{2}$ be a regular curve. Suppose that $\Sigma_{1}$ and $\Sigma_{2}$ intersect along $\Gamma$ at a constant angle. Assume also that
\begin{enumerate}
\item If both surfaces are tangent along $\Gamma$, then $N_{1} = N_{2}$ along $\Gamma$.
\item If the intersection is transversal along $\Gamma$, then their respective angle functions satisfy $\nu_{1} = -\nu_{2}$ along $\Gamma$.
\end{enumerate}

Then, $\Gamma$ is AR-line of curvature for $\Sigma_{1}$ if, and only if, $\Gamma$ is AR-line of curvature for $\Sigma_{2}$.
\end{Corollary}
\begin{proof}
Let  $S^{i}_{AR} X = A_{i}(X) - \alpha_{i} \langle \mt^{i}_{\theta_{i}}, X \rangle \mt^{i}_{\theta_{i}} + \frac{\alpha_{i} \abs{\mt_{i}}^{2}}{2}X-H_{i}X$ the Abresch-Rosenberg shape operator of $\Sigma_{i}$, $i=1,2$ and $J_{1}$, $J_{2}$ the rotations on the tangent bundles of $\Sigma_{1}$ and $\Sigma_{2}$ respectively.\\

In the first case, we have that $\mt^{1}_{\theta_{1}} \equiv \mt^{2}_{\theta_{2}}$ along $\Gamma$ since $\mt_{1}\equiv \mt_{2}$ and the surfaces has the same mean curvature. Moreover, if $\Gamma = \gamma(-\epsilon,\epsilon)$, therefore $J_{1}\gamma' = J_{2}\gamma'$ and so $II^{1}_{AR}(\gamma',J_{1}\gamma') = II^{2}_{AR}(\gamma',J_{2}\gamma')$.\\

Suppose now that we are in case 2, then $\theta_{1}=\theta_{2}=\theta$ and this implies $\alpha_{1}=\alpha_{2}=\alpha$, so
\begin{equation}
\label{eqt1}
\begin{split}
\alpha \la  \mt^{1}_{\theta} , \gamma'(t)  \ra \la \mt^{1}_{\theta},N_{2}\rangle  &= \alpha \la \mt_{1},\gamma'(t) \ra \la \mt_{1},N_{2} \ra \cos^{2}\theta \\
& \quad  + \alpha\cos\theta\sin\theta \la \mt_{1},\gamma'(t) \ra \la J_1\mt_{1},N_{2} \ra \\
 & \quad + \alpha\la J_1\mt_{1},\gamma'(t) \ra \la \mt_{1},N_{2} \ra \sin \theta \cos\theta \\
& \quad + \alpha \la J_1\mt_{1}, \gamma'(t) \ra \la J_1\mt_{1},N_{2} \ra \sin^{2}\theta.
\end{split}
\end{equation}

We oriented $\gamma$ such that $\{N_{1},N_{2},\gamma'(t) \}$ is an oriented basis of $T_{\gamma(t)}\hmf$ for each t were the intersection is transversal, then the following equations holds 
\begin{equation}
\label{eqt3}
\begin{split}
\la J_1\mt_{1}, \gamma'(t) \ra & = \la N_{1} \wedge \mt_{1},\gamma'(t) \ra = (1-d^2) \la N_{2},\mt_{1} \ra. \\
\la J_1\mt_{1}, N_{2} \ra & = \la N_{1} \wedge \mt_{1}, N_{2} \ra = -\frac{1}{1-d^2}\la \gamma'(t), \mt_{1} \ra.
\end{split}
\end{equation}

Hence using \eqref{eqt3} in the equation \eqref{eqt1}, we can rewrite the equation \eqref{eqt1} as

\begin{equation}
\label{eqt4}
\begin{split}
\alpha \la \mt^{1}_{\theta} , \gamma'(t)  \ra \la \mt^{1}_{\theta},N_{2}\rangle  &=  \alpha (\cos^{2}\theta -\sin^{2}\theta )\la \mt_{1},\gamma'(t) \ra \la \mt_{1},N_{2} \ra \\
& \quad  + \alpha \cos \theta\sin\theta (\la  \mt_{1}, N_{2} \ra^{2} - \la \mt_{1},\gamma'(t)\ra^{2} ). 
\end{split}
\end{equation}

Analogously, we obtain

\begin{equation}
\label{eqt7}
\begin{split}
\alpha \la \mt^{2}_{\theta} , \gamma'(t)  \ra \la \mt^{2}_{\theta},N_{1}\rangle  &=  \alpha ( \cos^{2}\theta-\sin^{2}\theta)  \la \mt_{2},\gamma'(t) \ra \la \mt_{2},N_{1} \ra  \\
& + \alpha \cos \theta\sin\theta (\la \mt_{2},\gamma'(t)\ra^{2} - \la  \mt_{2}, N_{1} \ra^{2}).  
\end{split}
\end{equation}

Now, the hypothesis implies that
\begin{equation} 
\label{angles}
\begin{split}
\la \mt_{1},N_{2} \ra&=\nu_{2}-\nu_{1}d = -\la \mt_{1},N_{2} \ra\\
\la \mt_{1},\gamma'(t)\ra&=\la\mt_{2},\gamma'(t)\ra=\la\xi, \gamma'(t)\ra. 
\end{split}
\end{equation} 

Finally, we substitutes the equations in \eqref{angles} in the equation \eqref{eqt4} and then from equation \eqref{eqt7}, we obtain   $$\sqrt{H^{2}+\tau ^{2}}\la \mt ^{2}_{\theta},N_{1} \ra \la J_{2} \mt ^{2}_{\theta},N_{1} \ra = \sqrt{H^{2}+\tau ^{2}}\la \mt ^{1}_{\theta},N_{2} \ra \la J_{1} \mt ^{1}_{\theta},N_{2} \ra,$$  
therefore, the Lemmma \ref{key} shows that $\Gamma$ is and AR-line of curvature of $\Sigma_{1}$ if, and only if $\Gamma$ is an AR- line of curvature of $\Sigma_{2}$.
\end{proof}

\begin{remark}
Note that we are assuming $H^{2}+\t^{2} \neq 0$. When $\t=0$, we consider the usual Abresch-Rosenberg shape operator. In \cite{dCF}, the authors studied lines of curvature with respect to the usual Abresch-Rosenberg shape operator and classified disks immersions in $\mer$. Their principal result is contained in Corollary \ref{CorodCF}, when $\t = 0$. Hence,  Lemma \ref{key} can be seen as a generalization to the case $\t \neq 0$. 
\end{remark}

\section{Capillary disks in $\hmf$.}
Throughout this section, we will denote by $\phi:\mathbb{D} \rightarrow \hmf$ an immersion from the disk $\mathbb{D}= \{z \in \mathbb{C}: \abs{z} < 1 \}$ onto $\hmf$ of constant mean curvature, we call it $H-$disk. Moreover, we will assume that the boundary $\Gamma$ is a smooth curve. 

\subsection{Immersed compact disks in $\hmf$, $\t=0$}

The classification of immersed compact disks with constant mean curvature in $\mer$ has been studied by M. Do Carmo and I. Fernández in \cite{dCF}, under certain conditions on the curve $\Gamma$, they showed that $\phi(\mathbb{D})$ is part of an Abresch-Rosenberg surface in $\mer$. In this section, we will classified immersed compact disks, assuming more general geometric conditions about $\Gamma$ than the conditions given in \cite{dCF}. \\

Let $\Omega$ be an Abresch-Rosenberg surface in $\mer$. We denote by $\nu_{1} = \la \xi,N_{1} \ra$ the angle function defined along the immersion $\phi$, where $N_{1}$ is the unit normal vector field defined along $\phi(\overline{\mathbb{D}})$ and by $\nu_{2} = \la \xi, N_{2} \ra$ the angle function defined along  $\Omega$, where $N_{2}$ is the unit normal vector field defined along surface $\Omega$.

\begin{theorem}
\label{Coro2}
Let $\phi:{\overline{\mathbb{D}}} \rightarrow \mer$ be a non minimal $H_{1}-$disk with regular boundary $\Gamma$. Suppose that $\phi$ meets transversally an Abresch-Rosenberg $H_{2}-$surface $\Omega$ along  $\Gamma$ at a constant angle.  Assume also that $\Gamma$ is of one of the following types:

\begin{enumerate}
\item $\Gamma$ is an horizontal curve.
\item  $\Gamma$ is a vertical curve of the immersion $\phi$ and the surface $\Omega$.
\item If $H_{1}=H_{2}$, and $\nu_{1} =-\nu _2$ along $\Gamma$. 
\end{enumerate}

Then, $\phi(\overline{\mathbb{D}})$ is part of an Abresch-Rosenberg surface in $\mer$. 
\end{theorem}

\begin{proof}
Set $\phi(\mathbb{S}^{1})=\Gamma$. Since $\Gamma$ is an AR-line of curvature of the Abresch-Rosenberg surface $\Omega$, then from the hypothesis and Corollary \ref{CoroDCF}, $\gamma$ is an AR-line of curvature of the immersion $\phi$. So, from Proposition \ref{A-R}, the imaginary part of the Abresch-Rosenberg differential must be zero along $\gamma$. Finally, the Schwarz Reflection Principle implies that the Abresch-Rosenberg differential must be zero on $\phi(\overline{\mathbb{D}})$ and Theorem \ref{QVa} gives the result. 
\end{proof}

\begin{example} In $\mathbb{H}^{2} \times \mathbb{R}$. Consider the totally geodesic plane $\Sigma$ parametrized as  
$$\psi(x,y) = (\cosh(x),0,\sinh(x),y)$$  where $x,y \in \mathbb{R}$ and the rotational sphere $\Omega$ with constant mean curvature $\frac{1}{\sqrt{2}}$ parametrized as
$$\varphi(u,v)=(\cosh(r(u)), \sinh(r(u))\cos(v), \sinh(r(u))\sin(v),h(u)) $$
where $(u,v) \in [-1,1] \times \mathbb{R}/2\pi$, 
 $r(u) = 2 {\rm arcsinh}(\sqrt{1-u^{2}})$ and $h(u)=\frac{4}{\sqrt{2}} {\rm arcsin}(\frac{\sqrt{2}}{2}u)$.\\

Observe that the surface $\Sigma$ is an example of a minimal surface whose Abresch-Rosenberg differential does not  vanish [see \cite{Rosen},section 4], moreover, a straightforward computation shows that  $$N_{\psi}(x,y)= (0,1,0,0)$$
and $$N_{\varphi}(u,v)= (h'(u) \sinh(r(u)), h'(u) \cos(v) \cosh(r(u)), h'(u) \cosh(r(u)) \sin(v), -r'(u))$$

are the normal vector fields of $\psi$ and $\varphi$ respectively. \\

Now, the intersection set $C =\Sigma \cap \Omega$ is the curve parametrized by $\gamma$ as
$$\gamma(s) = (\cosh(r(s)), 0 , \sinh(r(s))\sin(t), h(s)),$$ where $s \in [-1,1]$ and  $t=\frac{\pi}{2}$ or $t=\frac{3 \pi}{2}$. So, computing $\la N_{\psi}, N_{\varphi} \ra$ along the curve $\gamma$, we find that
$$ \la N_{\psi}, N_{\varphi} \ra |_{\gamma(s)} = h'(s) \cos(t) \cosh(r(s)) = 0.$$

Hence, along $\gamma$, the surfaces $\Sigma$ and $\Omega$ intersect at constant angle. \\

Finally, we define the surface $\Gamma$ that is the part of the totally geodesic plane $\Sigma$ inside of the rotational surface $\Omega$ after of the intersection $\Sigma \cap \Omega$. Then we have that $\partial \Gamma \equiv C$  and there exists an immersion $\phi:\mathbb{D} \rightarrow \Sigma$ such that 
\begin{itemize}
\item $\phi(\mathbb{D}) = \Gamma$,
\item $\phi(\mathbb{S}^{1})= C$ and
\item $\phi$ meets $\Omega$ along the curve $C$ at a constant angle,
\end{itemize}but $\phi$ is not a part of an Abresch-Rosenberg surface.\\
\end{example}

In Theorem \ref{Coro2} we have omitted two cases:

\begin{enumerate} 
\item  When $\Omega$ is a minimal Abresch-Rosenberg surface in $\mathbb{M}^{2}(\k) \times \mathbb{R}$ then $\Omega$ is a slice by Theorem \ref{QVa}. When $\Omega$ is a slice was studied in \cite[Theorem 9]{Cal}.
\item When the immersion $\phi$ is a minimal disk and $\Gamma$ is horizontal, one can solve this case using the Maximum principle, comparing $\phi$ with a slice. 
\end{enumerate}

\begin{remark}
Theorem \ref{Coro2} generalizes the classification result  given by Do Carmo and Fernandez in \cite[Corollary 4.1]{dCF} for immersed compact disks with constant mean curvature. We only assume that $\Gamma$ is horizontal, without assuming that $\Gamma$ is a line of curvature of the second fundamental form of the immersion.
\end{remark}

\subsection{Immersed compact disks in $\hmf$, $\t \neq 0$}
Now, we deal with $H-$ disks in $\hmf$, $\t  \neq 0$. We remember that for this class of immersions, we consider the Abresch-Rosenberg shape operator on $\phi(\mathbb{D})$ defined by formula \eqref{ARTraceless}. Then, using  Corollary \ref{CorodCF}, we extend the above classification result for the case $\t \neq 0$.

\begin{theorem}
\label{tNoZ}
Let $\phi: \mathbb{D} \rightarrow \hmf$, $\t \neq 0$, be a $H-$disk with regular boundary, suppose the boundary is parametrized by a regular curve $\gamma$ and it is of one of the following types
\begin{enumerate}
\item $\gamma$ is the tangent intersection of the immersion $\phi$ with an Abresch-Rosenberg surface $\Omega$ with the same mean curvature vector.
\item $\gamma$ is the transverse intersection with constant angle of the immersion $\phi$ with an Abresch-Rosenberg surface $\Omega$ with the same mean curvature and whose angle function is opposite to the angle function of the immersion $\phi$ along $\gamma$.
\end{enumerate}
Then, $\phi(\mathbb{D})$ is a part of an Abresch-Rosenberg surface in $\hmf$.
\end{theorem}

\begin{proof}
Set $\phi(\mathbb{S}^{1})=\Gamma$. Since $\Gamma$ is an AR-line of curvature of the Abresch-Rosenberg surface $\Omega$, then from hypothesis and Corollary \ref{CorodCF}, $\gamma$ is an AR-line of curvature of the immersion $\phi$. So, from Proposition \ref{A-R}, the imaginary part of the Abresch-Rosenberg differential must be zero along $\gamma$. Finally, the Schwarz Reflection Lemma implies that the Abresch-Rosenberg differential must be zero on $\phi(\overline{\mathbb{D}})$ and Theorem \ref{QVa} gives the result.
\end{proof}

\subsection{Immersed compact disks with non-regular boundary in $\hmf$.}
\label{CompaNoreg}

Now, we will study $H-$disks $\phi: \mathbb{D} \rightarrow \hmf$ with piece-wise regular boundary. Indeed, we suppose that $\phi(\mathbb{D})$ is contained in the interior of a differentiable surface without boundary in $\hmf$.\\

First, we will recall a result that gives conditions for a disk type surface to be umbilical with respect to a Codazzi pair with constant mean curvature.

\begin{theorem}[\cite{ER3}]
\label{ESP-FER}
Let $\Sigma$ be a compact disk with piece-wise smooth boundary. We will call the vertices of the surface to the finite set of non-regular boundary points. Assume that $\Sigma$ is contained as an interior set in a differentiable surface $\hat{\Sigma}$ without boundary.

Let $(I,II)$ be a Codazzi pair with constant mean curvature $H(I,II)$ on $\hat{\Sigma}$. Assume also that the following conditions holds:
\begin{enumerate}
\item The number of vertices in $\partial \Sigma$ with an angle $< \pi$ (measured with respect to the metric $I$) is less than 3.
\item The regular curves in $\partial \Sigma$ are lines of curvature for the pair $(I,II)$
\end{enumerate}
Then $\Sigma$ is totally umbilical for the pair $(I,II)$.
\end{theorem}

If $\Sigma$ is a $H-$ surface in $\hmf$, therefore the fundamental pair $(I,II_{S})$ defined by equation \eqref{Oper} when $\t=0$ and by equation \eqref{ARTraceless} when $\t \neq 0$ is a Codazzi pair and the mean curvature of pair $H(I,II_{S})$  is constant. Then, using Theorem \ref{ESP-FER} we obtain the following:

\begin{theorem}
\label{CoroSR}
Let $\Sigma$ be a $H-$ disk in $\hmf$ with piece-wise differentiable boundary. Assume also that the following conditions are satisfied:

\begin{enumerate}
\item $\Sigma$ is contained as an interior set in a smooth $H$-surface $\hat{\Sigma}$ in $\hmf$ without boundary.
\item The number of vertices in $\partial \Sigma$ with angle $<$ $\pi$ is less than or equal to 3.
\item The regular curves in $\partial \Sigma$ are AR-lines of curvatures of $\Sigma$.
\end{enumerate}
Then, $\Sigma$ is a part of an Abresch-Rosenberg surface in $\hmf$.
\end{theorem}

Theorem \ref{Coro2} shows that if a $H-$ disk in $\mer$ with horizontal curve $\Gamma$ as boundary meets an Abresch-Rosenberg surface at a constant angle along $\Gamma$, then $\Gamma$ is an AR- line of curvature of the immersion, hence Theorem \ref{CoroSR} implies the following corollary:

\begin{Corollary}
\label{NR1}
Let $\phi: \mathbb{D} \rightarrow \mer$ be a non minimal $H-$ disk, with piece-wise differentiable boundary $\Gamma$ that meets transversally an Abresch-Rosenberg $H_{2}-$ surface $\Omega$ along $\Gamma$ at a constant angle. Assume also that the following conditions are satisfied:

\begin{enumerate}
\item $\phi(\mathbb{D})$ is contained as an interior set in a smooth $H-$surface $\hat{\Sigma}$ in $\hmf$ without boundary.
\item The number of vertices in $\Gamma$ with angle $<$ $\pi$ is less than or equal to 3.
\item Every regular component $\gamma$ of $\Gamma$ is a one of following types: 
\begin{itemize}
\item $\gamma$  is a horizontal curve of $\Omega$.
\item $\gamma$ is a vertical curve of immersion $\phi$ and the surface $\Omega$.
\item If $H=H_{2}$ then the angle function of immersion $\phi$ is opposite to the angle function of $\Omega$ along $\gamma$. 
\end{itemize}
\end{enumerate}

Then, $\phi(\mathbb{D})$ is a part of an Abresch-Rosenberg surface in $\mer$.

\end{Corollary}

Theorem \ref{tNoZ}, shows that if a $H-$ disk in $\hmf$, $\t \neq 0$, has regular boundary, then under certain conditions over $\Gamma$, we conclude that $\Gamma$ is an AR- line of curvature of the immersion. So, using Theorem \ref{CoroSR}, we obtain:

\begin{Corollary}
\label{NR2}
Let $\phi: \mathbb{D} \rightarrow \hmf$, $\t \neq 0$, be a $H-$disk with piece-wise differentiable boundary $\Gamma$. Assume also that the following conditions are satisfied:

\begin{enumerate}
\item $\phi(\mathbb{D})$ is contained as an interior set in a smooth $H-$surface $\hat{\Sigma}$ in $\hmf$ without boundary.
\item The number of vertices in $\Gamma$ with angle $<$ $\pi$ is less than or equal to 3.
\item Every regular component $\gamma$ of $\Gamma$ is one of the following types:
\begin{itemize}
\item $\gamma$  is a tangent intersection of $\phi(\mathbb{D})$ with an Abresch-Rosenberg surface $\Omega$ with the same mean curvature vector. 
\item $\gamma$ is a transverse intersection with constant angle of $\phi(\mathbb{D})$ with an Abresch-Rosenberg surface $\Omega$ with the same constant mean curvature and whose angle function is opposite to the angle function $\phi(\mathbb{D})$ along $\gamma$. 
\end{itemize}
\end{enumerate}

Then, $\phi(\mathbb{D})$  is a part of an Abresch-Rosenberg surface in $\hmf$.\\

\end{Corollary}

The  author was supported by CAPES-Brazil and by CNPq-Brazil.\\

\end{document}